\documentclass[12pt]{amsart}

\usepackage[T1]{fontenc}

\usepackage{lmodern}

\usepackage{amsmath,amssymb,enumerate,color}

\usepackage[abbrev]{amsrefs}

\usepackage{fullpage}

\newtheorem{thm}{Theorem}

\newtheorem{prop}[thm]{Proposition}

\begin{document}

\title[Borderline
Bourgain-Brezis-Sobolev Inequalities]{Applications of
Bourgain-Brezis Inequalities to  Fluid Mechanics and Magnetism }

\author{Sagun Chanillo}
\address{Department of Mathematics, Rutgers, the State University of New Jersey, NJ 08854}
\email{chanillo@math.rutgers.edu}

\author{Jean Van Schaftingen}
\address{Institut de Recherche en Math\'ematique et en Physique\\ Universit\'e catholique de Louvain\\ Chemin du Cyclotron 2 bte L7.01.01\\ 1348 Louvain-la-Neuve\\ Belgium}
\email{Jean.VanSchaftingen@uclouvain.be}

\author{Po-Lam Yung}
\address{Department of Mathematics\\ the Chinese University of Hong Kong\\
Shatin\\ Hong Kong} \email{plyung@math.cuhk.edu.hk}

\thanks{S.C. was partially supported by NSF grant
DMS 1201474. J.V.S. was partially supported by the Fonds de la
Recherche Scientifique-FNRS. P.-L.Y. was  partially supported by a
direct grant for research from the Chinese University of Hong Kong
(4053120). We thank Haim Brezis for several comments that improved
the paper}

\begin{abstract} As a consequence of inequalities due to Bourgain-Brezis, we obtain local in time well-posedness for the two dimensional Navier--Stokes equation with velocity bounded in spacetime and initial vorticity in bounded variation. We also obtain spacetime estimates for the magnetic field vector through improved Strichartz inequalities.
\end{abstract}

\maketitle

\vspace{0.1 in}

% Our goal is to prove a theorem about the incompressible Navier--Stokes
% flow in $\mathbb{R}^2$. We also derive estimates for the magnetic
% field vector ${\bf B}$ in electromagnetism.

\section{Incompressible Navier--Stokes flow}

Let $\mathbf{v}(x,t)\in \mathbb{R}^2$ be the velocity and $p (x, t)$
be the pressure of a fluid of viscosity \(\nu > 0\) at position $x
\in \mathbb{R}^2$ and time $t \in \mathbb{R}$, governed by the
incompressible two-dimensional Navier--Stokes equation:
\begin{equation}\label{eq2}
\left\{
\begin{aligned}
\mathbf{v}_t+(\mathbf{v}\cdot \nabla)\mathbf{v}&=\nu\Delta {\bf
    v}-\nabla p,\\
    \nabla\cdot \mathbf{v}&=0,
\end{aligned}
\right.
\end{equation}
When the viscosity coefficient $\nu$ degenerates to zero,
\eqref{eq2} becomes the Euler equation. In two spatial dimensions, the vorticity of the flow is a scalar, defined by
$$
\omega= \partial_{x_1} v_2 - \partial_{x_2} v_1
$$
where we wrote $\mathbf{v}=(v_1,v_2)$.
%$$   \boldsymbol{\omega}=\nabla\times \mathbf{v}. $$
%We note that in two dimension, only the last component of the vorticity vector $\boldsymbol{\omega}$ could be non-zero.  We identify the vorticity with its last component, which we denote by $\omega$, that is $\boldsymbol{\omega}=(0,0,\omega)$, where $\omega= \partial_{x_1} v_2 - \partial_{x_2} v_1$, and where $\mathbf{v}=(v_1,v_2)$.
In the sequel, when we consider the Navier-Stokes equation, without loss of generality we set the viscosity coefficient $\nu=1$.

The vorticity associated to the incompressible Navier-Stokes flow in
two dimensions propagates according to the equation
\begin{equation}\label{eq9}
\omega_t -\Delta\omega=-\nabla\cdot(\mathbf{v}\omega).
\end{equation}
This follows from \eqref{eq2} by taking the curl of both sides.
We express the velocity $\mathbf{v}$ in the Navier-Stokes equation in terms of the vorticity
 through the Biot-Savart relation
%\begin{equation}\label{eq:vtoomega2}
%\mathbf{v} = (-\Delta)^{-1} (\nabla \times \boldsymbol{\omega}),
%\end{equation}
%where $\boldsymbol{\omega} = (0,0,\omega)$.
\begin{equation} \label{eq:BS2d}
\mathbf{v} = (-\Delta)^{-1} (\partial_{x_2} \omega, -\partial_{x_1} \omega).
\end{equation}
This follows formally by differentiating $\omega= \partial_{x_1} v_2 - \partial_{x_2} v_1$, and using that $\nabla \cdot \mathbf{v} = 0$.
%See for example Theorem II of Kato \cite{Kato1994} for a more careful justification for the Biot-Savart relation.
% and where the operator $(-\Delta)^{-1}\nabla\times$ is denoted as $S$.
%In the sequel the symbol $\nabla$ always will denote differentiation in the spatial variable $x$.

Our theorem states:
\begin{thm} \label{thm:fluid} Consider the two-dimensional
vorticity equation \eqref{eq9} and an initial vorticity $\omega_0
\in W^{1, 1} (\mathbb{R}^2)$ at time $t=0$. If
$$
\|\omega_0\|_{W^{1,1}(\mathbb{R}^2)}\leq A_0,
$$
then there exists a unique solution to the vorticity equation
\eqref{eq9} for all time $t\leq t_0=C/A_0^2$,  such that
             $$\sup_{t\leq t_0}\|\omega(\cdot,
             t)\|_{W^{1,1}(\mathbb{R}^2)}\leq cA_0.$$
Moreover, the solution $\omega$ depends continuously on the initial data $\omega_0$, in the sense that if $\omega_0^{(i)}$ is a sequence of initial data converging in $W^{1,1}(\mathbb{R}^2)$ to $\omega_0$, then the corresponding solutions $\omega^{(i)}$ to the vorticity equation \eqref{eq9} satisfies
$$
\sup_{t \leq t_0} \|\omega^{(i)}(\cdot,t)-\omega(\cdot,t)\|_{W^{1,1}(\mathbb{R}^2)} \to 0
$$
as $i \to \infty$.

Finally, the velocity vector $\mathbf{v}$ defined by the Biot-Savart relation (\ref{eq:BS2d}) solves the 2-dimensional incompressible Navier-Stokes (\ref{eq2}), and satisfies
            $$\sup_{t\leq t_0}\|{\mathbf
            {v}(\cdot,t)}\|_{L^\infty(\mathbb{R}^2)}+\sup_{t\leq t_0}\|\nabla{\mathbf
            v}(\cdot,t)\|_{L^2(\mathbb{R}^2)}\leq cA_0.$$
\end{thm}

Via the Gagliardo-Nirenberg inequality we note that we can conclude
from our theorem that
$$\sup_{0\leq t\leq t_0}\|\omega(\cdot,t)\|_{L^p(\mathbb{R}^2)}\leq C, \quad 1\leq p\leq 2.$$
In particular this is enough to apply Theorem II of Kato
\cite{Kato1994} to express the velocity vector in the Navier-Stokes
equation \eqref{eq2} in terms of the vorticity via the Biot-Savart
relation displayed above.

 In  \citelist{\cite{GigaMiyakawaOsada1988}\cite{Kato1994}},
it was proved that under the hypothesis that the initial vorticity
is a measure, there is a global solution that is well-posed to the
vorticity and Navier--Stokes equation; see also an alternative
approach in Ben-Artzi \cite{MR1308857}, and a stronger uniqueness
result in Brezis \cite{MR1308858}. The velocity constructed then
satisfies the estimate \cite{Kato1994}*{(0.5)}
\begin{equation}\label{eq13}
\|\mathbf{v}(\cdot,t)\|_{L^\infty(\mathbb{R}^2)}\leq
Ct^{-{\frac{1}{2}}},\ t\to 0.
\end{equation}
In contrast, in Theorem~\ref{thm:fluid} we have ${\bf
v}\in L^\infty_tL^\infty_x$, $x\in\mathbb{R}^2$, though we
are assuming that the initial vorticity has bounded variation, that is, its gradient
is a measure.

The estimate \eqref{eq13} is
indeed sharp as can be seen by the famous example of the \emph{Lamb--Oseen vortex} \cite{Oseen}, which consists of an initial vorticity
$\omega_0=\alpha_0\delta_0$, a Dirac mass at the origin of
$\mathbb{R}^2$ with strength $\alpha_0$. The constant $\alpha_0$ is
called the total circulation of the vortex. A unique solution to the
vorticity equation \eqref{eq9}
%with the coefficient of viscosity $\nu=1$
can be obtained by setting
         $$\omega(x,t)= {\frac{\alpha_0}{4\pi
         t}}e^{-{\frac{|x|^2}{4t}}},\ \quad {\bf
         v}(x,t)={\frac{\alpha_0}{2\pi}}{\frac{(-x_2,x_1)}{|x|^2}}\Big(1-e^{-{\frac{|x|^2}{4t}}}\Big).$$
It can be seen from the identities above that,
$$\|\omega(\cdot,t)\|_{W^{1,1}(\mathbb{R}^2)}\sim\|{\bf
v}(\cdot,t)\|_{L^\infty(\mathbb{R}^2)}\sim ct^{-{\frac{1}{2}}}, \
t\to 0.$$ Hence the assumption that the initial vorticity is a
measure cannot yield an estimate like in Theorem~\ref{thm:fluid}.
Thus to get uniform in time, $L^\infty$ space bounds all the way to
$t=0$ we need a stronger hypothesis and one such is vorticity in
BV (bounded variation).

It is also helpful to further compare our result with that of Kato
\cite{Kato1994} who establishes in (0.4) of his paper, that given
the initial vorticity is a measure, one has for the vorticity at
further time,
           $$||\nabla\omega(\cdot,t)||_{L^q(\mathbb{R}^2)}\leq
           ct^{\frac{1}{q}-\frac{3}{2}},\quad 1<q\leq \infty.$$
In contrast we obtain uniform in time bounds for $q=1$, as opposed
to singular bounds for $q>1$ when $t\to 0$.

It is an open question whether there is a global version of Theorem
1 of our paper.
\\[10pt]

In order to prove Theorem~\ref{thm:fluid}, we rely on a basic
proposition that follows from the work of Bourgain and Brezis
\citelist{\cite{BourgainBrezis2004}\cite{BourgainBrezis2007}}. A
part of this proposition also holds in three dimension. Recall that if
 $\mathbf{v}(x,t)\in \mathbb{R}^3$ be the velocity of a fluid at a point $x \in \mathbb{R}^3$ at time $t$, then the vorticity of $\mathbf{v}$ is defined by
$$
  \boldsymbol{\omega}=\nabla\times \mathbf{v}.
$$
Under the assumption that the flow is incompressible, the
Biot-Savart relation reads
\begin{equation}\label{eq:vtoomega2}
\mathbf{v} = (-\Delta)^{-1} (\nabla \times \boldsymbol{\omega}).
\end{equation}

%That is we also consider a vector field $\boldsymbol{\omega}\in\mathbb{R}^3$ and a vector field $\boldsymbol{v}\in \mathbb{R}^3$ related via the Biot-Savart relation \eqref{eq:vtoomega2}. However it is part (b) that is most useful to us in the sequel.
\begin{prop} \label{prop}
%Assume that $\mathbf{v}$ is the velocity field for any flow  in 2 or 3 spatial dimensions, which satisfies equation \eqref{eq:vtoomega2}.
\begin{enumerate}[(a)]
\item \label{prop_a} Consider the velocity $\mathbf{v}$ in 3 spatial dimensions. Assume that $\mathbf{v}$ satisfies the Biot-Savart relation \eqref{eq:vtoomega2}. Then at any fixed time $t$,
$$\|\mathbf{v}(\cdot,t)\|_{L^3(\mathbb{R}^3)}+\|\nabla\mathbf{v}(\cdot,t)\|_{L^{3/2}(\mathbb{R}^3)}\leq
                 C\|\nabla\times \boldsymbol{\omega}(\cdot,t)\|_{L^1(\mathbb{R}^3)}$$
where $C$ is a constant independent of $t$, $\mathbf{v}$ and $\boldsymbol{\omega}$.
\item \label{prop_b} Consider the velocity $\mathbf{v}$  in 2 spatial dimensions. Assume that $\mathbf{v}$ satisfies the Biot-Savart relation  (\ref{eq:BS2d}). Then  at any fixed time $t$,
       $$\|\mathbf{v}(\cdot,t)\|_{L^\infty(\mathbb{R}^2)}+\|\nabla \mathbf{v}(\cdot,t)\|_{L^2(\mathbb{R}^2)}\leq
       C \|\nabla \omega(\cdot,t)\|_{L^1(\mathbb{R}^2)}.$$
where $C$ is a constant independent of $t$, $\mathbf{v}$ and $\omega$.
\end{enumerate}
\end{prop}

%{\color{red} Here $\|\omega\|_{\dot{W}^{1,1}(\mathbb{R}^2)} := \|\nabla \omega\|_{L^1(\mathbb{R}^2)}$.}
We remark that in 2 dimensions, by the Poincar\'{e} inequality, it follows from $\|\nabla\mathbf{v}\|_{L^2(\mathbb{R}^2)}<\infty$, that
$\mathbf{v}$ lies in $VMO(\mathbb{R}^2)$, i.e. has vanishing mean oscillation.

\begin{proof}[Proof of Proposition~\ref{prop}]
 Note that
$$\nabla\cdot (\nabla\times \boldsymbol{\omega})=0.$$ Thus we
can immediately apply the result of Bourgain-Brezis
\cite{BourgainBrezis2007} (see also
\citelist{\cite{BourgainBrezis2004}\cite{CvSY}\cite{VanSchaftingen2004}}),
to the Biot-Savart formula \eqref{eq:vtoomega2} and get the desired
conclusions in part \eqref{prop_a}.

To consider the 2-dimensional flow,  note that
$(-\partial_{x_2}\omega,\partial_{x_1}\omega)$ is a vector field in
$\mathbb{R}^2$ with vanishing divergence. In view of  the
two-dimensional  Biot-Savart relation (\ref{eq:BS2d}), we can then
use the two-dimensional Bourgain--Brezis result
\cite{BourgainBrezis2007}, and we obtain \eqref{prop_b}.
\end{proof}

%In two dimensions we have an alternative way to state part \eqref{prop_b} of Proposition~\ref{prop}. Since vorticity can be viewed as a scalar function in two dimensions, we have,                $$\|\nabla\times\omega\|_{L^1(\mathbb{R}^2)}=\|\nabla\omega\|_{L^1(\mathbb{R}^2)}=\|\omega\|_{\dot{W}^{1,1}(\mathbb{R}^2)}.$$

%\begin{cor}\label{cor}
%For an incompressible flow in two dimensions, we have $$ \|\mathbf{v}\|_{L^\infty(\mathbb{R}^2)}+\|\nabla \mathbf{v}\|_{L^2(\mathbb{R}^2)}\leq  C\|\omega\|_{\dot{W}^{1,1}(\mathbb{R}^2)}.$$
%\end{cor}

We note further that the proposition applies to both the Euler
(inviscid) or the Navier--Stokes (viscous) flow.

%\omega_t -\Delta\omega=-\partial_{x_1}(\omega (-\Delta)^{-1}  \partial_{x_2}  \omega) + \partial_{x_2}(\omega (-\Delta)^{-1} \partial_{x_1}  \omega).

\begin{proof}[Proof of Theorem~\ref{thm:fluid}]

Now set $K_t$ for the heat kernel in 2-dimensions, i.e.
          $$K_t(x) = {\frac{1}{4\pi t}}e^{-{\frac{|x|^2}{4t}}}.$$
Rewriting \eqref{eq9}  as an integral equation for $\omega$ using
Duhamel's theorem, where $\omega_0$ is the initial vorticity, we
have,
\begin{equation}\label{eq10}
\omega(x,t)=K_t\star\omega_0(x)+\int_0^t\partial_x K_{t-s}\star[\mathbf{v} \omega(x,s)]ds
\end{equation}
where $\mathbf{v}$ is given by \eqref{eq:BS2d}.

We apply a Banach fixed point argument to the operator $T$ given by
\begin{equation}\label{eq11}
T\omega(x,t)=K_t\star\omega_0(x)+\int_0^t\partial_x K_{t-s}\star[\mathbf{v} \omega(x,s)] ds,
\end{equation}
where again $\mathbf{v}$ is given by \eqref{eq:BS2d}.
Let us set
          $$E=\Bigl\{g\ | \sup_{0<t<t_0} \|g(\cdot,t)\|_{W^{1,1}(\mathbb{R}^2)}\leq A\Bigr\}.$$
We will first show that $T$ maps $E$ into itself, for $t_0$ chosen
as in the theorem.

Differentiating \eqref{eq11} in the space variable once, we get
 $$
 (T\omega(x,t))_x=K_t\star f_0(x)+\int_0^t \partial_x K_{t-s}\star \big( {\bf
 v}_x \omega\big)ds+\int_0^t \partial_x K_{t-s}\star \big(\mathbf{v} \omega_x\big)ds.$$
Here we denote by $f_0$ the spatial derivative of the initial vorticity $\omega_0$.
%,and we have set $\mathbf{v}=(-\Delta)^{-1}(\nabla\times \omega)$.
Using Young's convolution inequality, we have
        $$\|(T\omega(\cdot,t))_x\|_{L^1(\mathbb{R}^2)}\leq \|f_0\|_{L^1(\mathbb{R}^2)}+C\int_0^t
        (t-s)^{-1/2} (\|\mathbf{v}_x\omega\|_{L^1(\mathbb{R}^2)}+\|\mathbf{v}\omega_x\|_{L^1(\mathbb{R}^2)})ds.$$
Now we apply Proposition~\ref{prop}\eqref{prop_b} to each of the terms on the right. For
the first term we have, by Cauchy-Schwartz,
$$
\|\mathbf{v}_x\omega\|_{L^1(\mathbb{R}^2)}\leq C\|\nabla \mathbf{v}\|_{L^2(\mathbb{R}^2)}\|\omega\|_{L^2(\mathbb{R}^2)}
$$
The Gagliardo-Nirenberg inequality applies as $\omega\in E$ and so
$\omega(\cdot, t)\in L^1(\mathbb{R}^2)$ and so,
            $$\|\omega\|_{L^2(\mathbb{R}^2)}\leq
            C\|\nabla\omega\|_{L^1(\mathbb{R}^2)},$$
and to $\|\nabla\mathbf{v}\|_{L^2(\mathbb{R}^2)}$ we apply Proposition~\ref{prop}\eqref{prop_b}.
%We see that each individual term on the right above in \eqref{eq12} is bounded by $\|\nabla\omega\|_1$.
Similarly
$$\|\mathbf{v}\omega_x\|_{L^1(\mathbb{R}^2)} \leq \|\mathbf{v}\|_{L^{\infty}(\mathbb{R}^2)} \|\omega_x\|_{L^1(\mathbb{R}^2)}.$$
Again we apply Proposition~\ref{prop}\eqref{prop_b} to $\|\mathbf{v}\|_{L^{\infty}(\mathbb{R}^2)}$. Hence in all we have,
              $$\|(T\omega)_x\|_{L^1(\mathbb{R}^2)} \leq \|f_0\|_{L^1(\mathbb{R}^2)} +C\int_0^t(t-s)^{-1/2}\|\nabla\omega\|^2_{L^1(\mathbb{R}^2)} ds.$$
Thus setting
$\|f_0\|_{L^1(\mathbb{R}^2)}=\|\omega_0\|_{\dot{W}^{1,1}(\mathbb{R}^2)}\leq
A_0$, we get for $t\leq t_0$ and since $\omega\in E$,
        $$ \|\nabla(T\omega)(\cdot,t)\|_{L^1(\mathbb{R}^2)} \leq A_0+Ct_0^{1/2} A^2.$$
Next from Young's convolution inequality it follows from
\eqref{eq11} that,
            $$\|T\omega(\cdot,t)\|_{L^1(\mathbb{R}^2)}\leq
            A_0+\int_0^t(t-s)^{-1/2}\|{\bf v}\omega(\cdot,s)\|_1ds$$
But by Proposition 2(b) again,
              $$||{\bf v}\omega||_1\leq ||{\bf
              v}||_\infty||\omega||_1\leq cA^2.$$
Thus,
          $$||T\omega(\cdot,t)||_1\leq A_0+ct^{1/2}A^2.$$

So adding the estimates for $T\omega$ and $\nabla(T\omega)$ we have,
         $$\sup_{t\leq
         t_0}||T\omega(\cdot,t)||_{W^{1,1}(\mathbb{R}^2)}\leq
         2A_0+ct_0^{1/2}A^2.$$

By choosing $A$ so that $A_0=A/8$ and $t<t_0=C/A_0^2$  we can assure
that if $\omega\in E$, then
              $$\sup_{t\leq t_0}\|(T\omega)(\cdot,t)\|_{W^{1,1}(\mathbb{R}^2)} \leq {\frac{A}{2}}.$$
Thus $T\omega\in E$, if $\omega\in E$. If we establish that $T$ is a
contraction then we are done.

Next we observe that the estimates in Proposition~\ref{prop}\eqref{prop_b} are linear estimates.
That is
      $$\|\mathbf{v}_1-\mathbf{v}_2\|_\infty+\|\nabla \mathbf{v}_1-\nabla{\bf
      v}_2\|_2\leq C\|\omega_1-\omega_2\|_{W^{1,1}(\mathbb{R}^2)}.$$
We easily can see from the computations above, that we have
          $$\sup_{t\leq t_0}\|T\omega_1-T\omega_2\|_{W^{1,1}(\mathbb{R}^2)}\leq
           CAt_0^{1/2} \sup_{t\leq t_0}\|\omega_1-\omega_2\|_{W^{1,1}(\mathbb{R}^2)}.$$
By the choice of $t_0$, it is seen that $T$ is a contraction. Thus
using the Banach fixed point theorem on $E$, we obtain our operator
$T$ has a fixed point and so the integral equation \eqref{eq10} has
a solution in $E$. The remaining part of our theorem follows easily
from  Proposition~\ref{prop}\eqref{prop_b}.
\end{proof}

We note in passing an estimate in $\mathbb{R}^3$ from Proposition \ref{prop}\eqref{prop_a}
above for the Navier--Stokes or the Euler flow:
\begin{equation}\sup_{t>0}\|{\bf
                 v}\|_{L^3(\mathbb{R}^3)}+\sup_{t>0}\|\nabla\mathbf{v}\|_{L^{3/2}(\mathbb{R}^3)}\leq
                 C\sup_{t>0}\|\nabla\times\boldsymbol{\omega}\|_{L^1(\mathbb{R}^3)}.
\end{equation}

\section{Magnetism}

We next turn to our results on magnetism. We denote by ${\bf
B}(x,t)$ and ${\bf E}(x,t)$ the magnetic and electric field vectors at $(x,t)\in\mathbb{R}^3\times
\mathbb{R}$. Let  ${\bf j}(x,t)$ denote the current density vector.
The Maxwell equations imply
\begin{align}
\nabla \cdot {\bf B} & = 0, \label{maxwell1} \\
\partial_t {\bf B} + \nabla \times {\bf E} & = 0, \label{maxwell2} \\
\partial_t {\bf E} - \nabla \times {\bf B} & = -{\bf j}. \label{maxwell3}
\end{align}
Differentiating \eqref{maxwell2} in $t$ and using \eqref{maxwell3}, together with the vector identity $\nabla \times (\nabla \times {\bf B}) = \nabla (\nabla \cdot {\bf B}) - \Delta {\bf B}$ and \eqref{maxwell1}, one obtains an inhomogeneous wave equation for ${\bf B}$:
\begin{equation}\label{eq14}
     {\bf B}_{tt}-\Delta {\bf B}=\nabla\times {\bf j}.
     \end{equation}
The right side of \eqref{eq14} satisfies the vanishing divergence
condition  $$\nabla\cdot(\nabla\times{\bf j})=0$$ for any fixed time
$t$. Thus an improved Strichartz estimate, namely Theorem 1 in \cite{ChanilloYung2012}
applies. We point out that the Bourgain-Brezis inequalities play a key
role in the proof of Theorem 1 in \cite{ChanilloYung2012}. We conclude easily:

\begin{thm} Let ${\bf B}$ satisfy \eqref{eq14} and let ${\bf B}(x,0)={\bf B}_0$,
$\partial_t{\bf B}(x,0)={\bf B_1}$ denote
the initial data at time $t=0$. Let $s,k\in \mathbb{R}$. Assume
$2\leq q\leq \infty$, $2<\tilde{q}\leq \infty$ and $2\leq r<\infty$.
Let $(q,r)$ satisfy the wave compatibility condition
$$
                    \frac{1}{q}+\frac{1}{r}\leq \frac{1}{2},
$$
and the following scale invariance condition is verified:
$$
      \frac{1}{q}+\frac{3}{r}=\frac{3}{2}-s=\frac{1}{\tilde
      {q}^\prime}+1-k
$$
Then, for $\frac{1}{\tilde{q}}+\frac{1}{\tilde{q}^\prime}=1$, we
have
 $$
      \|{\bf B}\|_{L^q_tL_x^r}+\|{\bf
      B}\|_{C_t^0\dot{H}_x^s}+||\partial_t {\bf
      B}\|_{C_t^0\dot{H}_x^{s-1}} \leq C(||{\bf B}_0\|_{\dot
      {H}^s}+\|{\bf
      B}_1\|_{\dot{H}^{s-1}}+\|(-\Delta)^{k/2}(\nabla_x{\bf
      j})\|_{L^{\tilde{q}^\prime}_tL_x^1}).
 $$
\end{thm}

The main point in the theorem above is that we have $L^1$ norm in space on the right side.

\end{document}